\documentclass [11pt,oneside]{amsart}
\usepackage{amsmath}

\usepackage {amsfonts}
\usepackage {amsmath}
\usepackage {amsthm}
\usepackage {amssymb}
\usepackage {framed}
\usepackage {amsxtra}
\usepackage {enumerate}
\usepackage {graphicx}
\usepackage{color}
\usepackage{graphicx}
\usepackage{graphicx}
\usepackage{wrapfig}
\usepackage{lscape}
\usepackage{rotating}
\usepackage{epstopdf}
\usepackage{pdflscape}
\usepackage{setspace}
\usepackage[left=0.86in, right=.86 in, top=1.2 in, footskip=0.5 in, bottom=0.9 in]{geometry}
\usepackage{adjustbox}
\makeatletter
\usepackage{empheq}

\newcommand*\widefbox[1]{\fbox{\hspace{2em}#1\hspace{2em}}}

\theoremstyle{definition}
\newtheorem{df}{Definition} [section]

\theoremstyle{plain}
\newtheorem{thm}[df]{Theorem}

\newtheorem{lemma}[df]{Lemma}

\newtheorem{obs}[df]{Observation}
\newtheorem{problem}[df]{Problem}

\title{Dense Packings with Nonparallel Cylinders}
\author{Dan Ismailescu and Piotr Laskawiec, Hofstra University}
\begin{document}

\begin{abstract}
A \emph{cylinder packing} is a family of congruent infinite circular cylinders with mutually disjoint interiors in $3$-dimensional Euclidean space.
The \emph{local density} of a cylinder packing is the ratio between the volume occupied by the cylinders within a given sphere and the volume of the entire sphere.
The \emph{global density} of the cylinder packing is obtained by letting the radius of the sphere approach infinity.

It is known that the greatest global density is obtained when all cylinders are parallel to each other and each cylinder is surrounded by exactly six others. In this case, the global density of the cylinder packing equals $\pi/\sqrt{12}= 0.90689\ldots$.
The question is how large a density can a cylinder packing have if one imposes the restriction that \emph{no two cylinders are parallel}.

In this paper we prove two results. First, we show that there exist cylinder packings with no two cylinders parallel to each other, whose local density is arbitrarily close
to the local density of a packing with parallel cylinders. Second, we construct a cylinder packing with no two cylinders parallel to each other whose global density is $1/2$.
This improves the results of K. Kuperberg, C. Graf and P. Paukowisch.

\end{abstract}
\maketitle
\pagenumbering{arabic}


\section{\bf Introduction}

Let $l$ be a line in $3$-dimensional space, and let $r$ be a positive real number. The set of points in $\mathbf{R}^3$ whose distance from $l$ is
not greater than $r$ is called the \emph{infinite circular cylinder} (or just the \emph{cylinder}, for short) of \emph{axis} $l$ and \emph{radius} $r$.
Two cylinders are \emph{parallel} if their axes are parallel. A collection of congruent cylinders whose interiors are mutually disjoint is called a
\emph{cylinder packing}.

Given a cylinder packing, 
$\mathcal{C}=\{C_i\}_{i=1}^{\infty}$,
one may inquire what portion of the space is occupied by the cylinders in $\mathcal{C}$.
This question naturally leads to the concepts of local density and global density of $\mathcal{C}$ defined below.

Let $B(R)$ be the ball of radius $R$ centered at the origin. Then the \emph{local density} of $\mathcal{C}$ with respect to $B(R)$, denoted $\delta(\mathcal{C},R)$, is
the ratio between the volume within $B(R)$ that is covered by the cylinders, and the volume of $B(R)$. More formally,
\begin{equation*}
\delta(\mathcal{C},R)=\frac{\sum\limits_{i=1}^{\infty}Vol(C_i \cap B(R))}{4\pi R^3/3}.
\end{equation*}
The \emph{global density} of $\mathcal{C}$, denoted $\delta(\mathcal{C})$, is simply
\begin{equation*}
\delta(\mathcal{C})=\lim_{R\rightarrow \infty}\delta(\mathcal{C},R),
\end{equation*}
provided the limit exits.

One may ask what is the greatest possible value of $\delta(\mathcal{C})$. In lay terms, what is the most efficient way to pack congruent cylinders in space?

Under the restriction that all cylinders are parallel, the problem is equivalent to finding the maximum density of a packing of the plane with congruent circles.
One can easily see this by writing $\mathbf{R}^3$ as the union of the planes perpendicular to the cylinders' axes.

The planar problem was solved by Thue \cite{AT}, who proved that the most efficient way to pack congruent circles in the plane is to arrange them such that each circle
is tangent to six others - see figure \ref{figthue}.
\vspace{-0.3cm}
\begin{figure}[ht]
\centering
\includegraphics[width=0.45\linewidth]{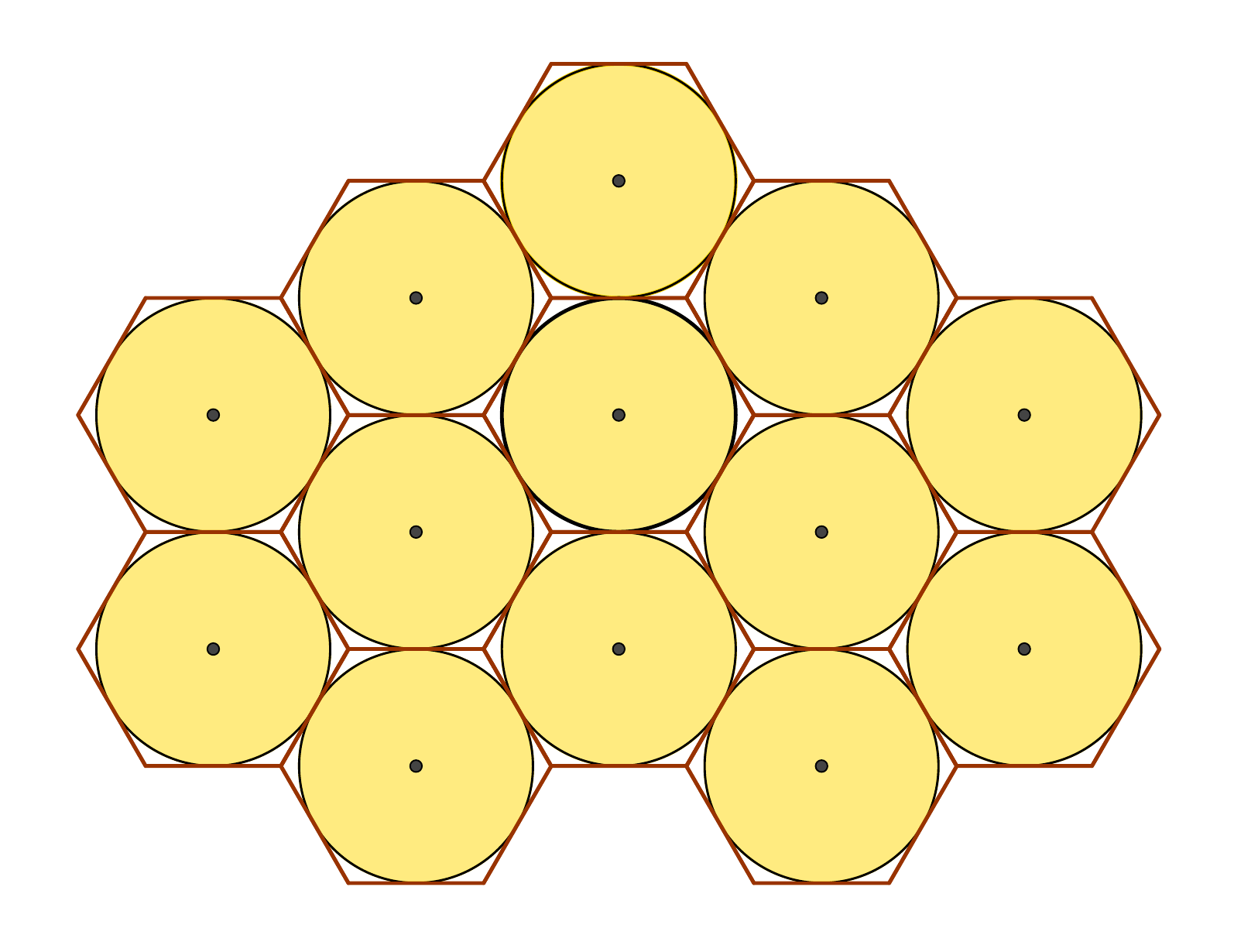}
\vspace{-0.4cm}
\caption{The densest packing of congruent circles in the plane has density $\pi/\sqrt{12}$}
\label{figthue}
\end{figure}

The density of such arrangement is given by the ratio between the area of a circle and the area of the regular circumscribed hexagon, and it equals $\pi/\sqrt{12}=0.90689\ldots$.

It follows that if all cylinders are parallel, then $\delta(\mathcal{C})\le \pi/\sqrt{12}$.

Can one do better? This disarmingly simple question remained unsolved until 1990 when A. Bezdek and W. Kuperberg proved in \cite{BK} that the answer is negative.

\begin{thm}\cite{BK} The maximum global density of any space packing with congruent cylinders is $\pi/\sqrt{12}$.
\end{thm}

This result is significant because this was the first example of a solid that does not tile the space, and whose maximum packing density was explicitly determined.

This prompted Moser and Pach to raise the following challenge:
\begin{problem}\cite{MP}
Find nontrivial examples of \underline{bounded} convex bodies in $\mathbf{R}^3$ that are not tiles  but whose maximum packing density can be computed exactly.
\end{problem}
A first answer was provided by Bezdek \cite{B}, who determined the maximum packing density of a rhombic dodecahedron that has one corner removed
so that it no longer tiles the space. In this case, the packing density equals the ratio of the volume of the truncated rhombic dodecahedron to the
volume of the full rhombic dodecahedron.

A few years later, Hales settled a long-standing conjecture of Kepler regarding the maximum packing density of the sphere. Hales showed that this value equals $\pi/\sqrt{18}=0.7404\ldots$. Hales's proof relies heavily on computers, and for this particular reason, a certain level of uncertainty about the validity of the proof
lingered within the mathematical community. In order to remove any such doubts, Hales embarked on a 15-year collaborative project to produce a complete formal proof of the Kepler conjecture. In January 2015, Hales and 21 collaborators submitted a paper titled \emph{A formal proof of the Kepler conjecture} \cite{H} to arXiv, claiming to have proved the conjecture. In 2017, the paper appeared in print in the \emph{Forum of Mathematics} journal.

Returning to the cylinder case, it is not obvious at first glance whether there exists a cylinder packing that does not contain two parallel cylinders but has positive density.
Under the restriction that no two cylinders are parallel, one might expect that the cylinders need so much room to avoid each other that in effect the packing density always
turns out to be zero.

Surprisingly, this is not true. A subtle construction of K. Kuperberg \cite{KK} produces a cylinder packing with density $\pi^2/576=0.017\ldots$, in which no two cylinders are parallel. Her construction was refined by Graf and Paukowitsch \cite{GP}, who obtained a nonparallel cylinder packing with density $5/12=0.4166\ldots$.

In this paper, we prove two results.
First, we show that there exist cylinder packings with no two cylinders parallel, whose local density is arbitrarily close to the local density of a cylinder packing with parallel cylinders. Second, we construct a nonparallel cylinder packing whose global density is $1/2$, thus improving Graf's and Paukowitch's result.

\section{\bf A local result}
For any cylinder packing in space, there exists a plane which intersects the axes of all the cylinders. This is because the number of cylinders in any packing is countable,
and therefore, so is the number of unit vectors representing the directions of their axes. Since the unit sphere cannot be expressed as a countable union of great circles, one can select a unit vector which is not orthogonal to any of these direction vectors. Then, any plane having that vector as its normal vector will have the desired property.

Thus, every cylinder packing can be constructed in the following manner:

\noindent{\bf Step 1.} Start with a point set $\mathcal{A}$ in the $xy$ plane.\\
{\bf Step 2.} Through each point $A\in \mathcal{A}$, construct a line $l\in \mathbf{R}^3$.\\
{\bf Step 3.} Consider the cylinders of radius $r$ whose axes are the lines constructed in step 2. If the distance $\phantom{asfsaf}$ between any two such lines is $\ge 2r$, then the resulting collection of cylinders is a packing.

Let $\mathcal{A}=\{A_1,A_2,\ldots,A_n,\ldots\}$ be a set of points in the $xy$ plane such that the distance between any two distinct points of $\mathcal{A}$ is at least $2r$, where $r>0$ is given. For each $A_i\in \mathcal{A}$, let $l_i^{\perp}$ be the line passing through $A_i$ and perpendicular to the $xy$ plane.

Further, let $C_i^{\perp}(r)$ be the cylinder of radius $r$ whose axis is $l_i^{\perp}$.
Then $\{C_i^{\perp}(r)\}$ is a cylinder packing with all cylinders parallel to each other. For any $R>0$, denote the local density of this packing by
\begin{equation}
\delta^{\perp}(\mathcal{A},r,R)=\frac{\sum\limits_{i=1}^{|\mathcal{A}|}Vol\left(C_i^{\perp}(r) \cap B(R)\right)}{4\pi R^3/3},
\end{equation}
where $B(R)$ is the ball of radius $R$ centered at the origin.

For every $A_i\in \mathcal{A}$, $A\neq O$, let $l_i$ be the image of $l_i^{\perp}$ under a rotation around the axis $OA_i$. Note that irrespective of the rotation angle,
$l_i$ is also perpendicular to $OA_i$. 

For some $0< \rho \le r$, let $C_i(\rho)$ be the cylinder of radius $\rho$ having line $l_i$ as its axis. Under the assumption that the distance
between any two lines $l_i$ and $l_j$ is at least $2\rho$, then $\{C_i(\rho)\}$ is also a cylinder packing.
Denote the local density of this new packing by
\begin{equation}
\delta(\mathcal{A},\rho, R)=\frac{\sum\limits_{i=1}^{|\mathcal{A}|}Vol\left(C_i(\rho) \cap B(R)\right)}{4\pi R^3/3}.
\end{equation}

What we intend to prove is that the lines ${l_i}$ can be chosen is such a way that no two cylinders in $\{C_i(\rho)\}$ are parallel, and the local density
$\delta(\mathcal{A},\rho, R)$ is not too much smaller than the local density of the parallel packing, $\delta^{\perp}(\mathcal{A},r,R)$.
This will be made precise in the sequel.

We start with a

\begin{lemma}\label{lemma1} Let $r>0, R>0$ be two fixed positive reals. Consider $A_1(x_1,y_1,0)$ and $A_2(x_2,y_2,0)$, two distinct points in the $xy$-plane such that $\|A_1A_2\|\ge 2r$. Let $O(0,0,0)$ be the origin, and denote $\|OA_1\|=d_1$, $\|OA_2\|=d_2$. Assume that $0<d_1\le R$, and $0<d_2\le R$.  For $i=1, 2$, let $l_i$ be the line passing through $A_i$ and having direction vector $\mathbf{v}_i=\langle y_i, -x_i, T \rangle$ where $T$ is chosen such that $8r^2T\ge R^4$. Then the lines $l_1$ and $l_2$ are not parallel, and the distance between them is at least
\begin{equation}\label{yejinjohn}
dist(l_1,l_2)\ge 2r\left(1-\frac{1}{T}\right)
\end{equation}
\end{lemma}

\begin{proof}
Note first that the angles formed by the lines $l_1, l_2$ with the $xy$ plane are $\arctan(T/d_1)$ and $\arctan(T/d_2)$, respectively. So, if $d_1\neq d_2$, the lines cannot be parallel.

If $d_1=d_2$, then $\|\mathbf{v}_1\|=\|\mathbf{v}_2\|$, and under the assumption that $l_1\parallel l_2$, we immediately obtain that $\mathbf{v}_1=\mathbf{v}_2$. But this then leads to $A_1=A_2$, contradiction. Hence, the lines cannot be parallel.

We denote by $c$, the cosine of the angle determined by the vectors $\overrightarrow{OA_1}$ and $\overrightarrow{OA_2}: c=\cos(\angle A_1OA_2)$.
Note that
\begin{equation*}
c=\frac{\overrightarrow{OA_1}\cdot \overrightarrow{OA_2}}{\|OA_1\|\|OA_2\|}=\frac{x_1x_2+y_1y_2}{d_1d_2}\quad \text{from which}\quad  x_1x_2+y_1y_2= c\,d_1d_2.
\end{equation*}
Also, by the law of cosines
\begin{equation*}
\|A_1A_2\|^2= d_1^2+d_2^2-2d_1d_2\cos(\angle A_1OA_2)= d_1^2+d_2^2-2cd_1d_2.
\end{equation*}
It is well known that the distance between two skew lines can be computed with the formula
\begin{equation}\label{dist1}
dist(l_1,l_2)=\frac{|\overrightarrow{A_1A_2}\cdot(\mathbf{v}_1 \times\mathbf{v}_2)|}{\|\mathbf{v}_1 \times \mathbf{v}_2\|}.
\end{equation}
We will express the numerator and denominator separately. For the numerator, we have
\begin{align}
\overrightarrow{A_1A_2}\cdot(\mathbf{v}_1 \times\mathbf{v}_2)&=
\begin{vmatrix}
x_2-x_1 & y_2-y_1 & 0 \\
y_1 &  -x_1& T \\
y_2 & -x_2 & T
\end{vmatrix}
=-T
\begin{vmatrix}
x_2-x_1 & y_2-y_1 \\
y_2 & -x_2
\end{vmatrix}+
T
\begin{vmatrix}
x_2-x_1 & y_2-y_1 \\
y_1 & -x_1
\end{vmatrix}=\notag\\
&= T\left( x_1^2+y_1^2+x_2^2+y_2^2-2x_1x_2-2y_1y_2\right)= T(d_1^2+d_2^2 - 2\overrightarrow{OA_1}\cdot \overrightarrow{OA_2}) =\notag\\
&= T\left(d_1^2+d_2^2-2d_1d_2\,\cos{\angle A_1OA_2}\right)= T\,\|A_1A_2\|^2.\label{top1}
\end{align}

Computing the denominator, it follows that
\begin{align}
\lVert\mathbf{v}_1\times\mathbf{v}_2\rVert^2
&=\lVert\mathbf{v}_1\rVert^2\lVert\mathbf{v}_2\rVert^2
-\|\mathbf{v}_1\cdot\ \mathbf{v}_2\|^2=\notag\\
&= (d_1^2 + T^2)(d_2^2 + T^2)-(x_1x_2 + y_1y_2 + T^2)^2=\notag\\
&=(d_1^2 + T^2)(d_2^2 + T^2)-(c\,d_1d_2 + T^2)^2=\notag\\
&= d_1^2d_2^2 + d_1^2T^2 + d_2^2T^2 + T^4 - c^2d_1^2d_2^2 - 2cd_1d_2T^2 -T^4 =\notag\\
&= (d_1^2 + d_2^2 - 2cd_1d_2)T^2 + (1- c^2)d_1^2d_2^2)\notag=\\
&= T^2\|A_1A_2\|^2+(1-c^2)d_1^2d_2^2.\label{bot1}
\end{align}
Using \eqref{top1} and \eqref{bot1} into \eqref{dist1}, we obtain that
\begin{align*}
\frac{1}{dist(l_1,l_2)^2}&=\dfrac{\|\mathbf{v}_1 \times\mathbf{v}_2\|^2}{\|A_1A_2\cdot(\mathbf{v}_1\times\mathbf{v}_2)\|^2}
=\frac{T^2\|A_1A_2\|^2+(1-c^2)d_1^2d_2^2}{T^2\,\|A_1A_2\|^4}\notag\\
&=\frac{1}{\|A_1A_2\|^2}+\frac{(1-c^2)d_1^2d_2^2}{T^2\,\|A_1A_2\|^4}
\leq \frac{1}{\|A_1A_2\|^2}+\frac{R^4}{T^2\,\|A_1A_2\|^4}\le \notag\\
&\le \frac{1}{4r^2}+\frac{8r^2T}{16r^4T^2}\le \frac{1}{4r^2}\left(1+\frac{1}{T}\right)^2,
\end{align*}
from which
\begin{equation*}
dist(l_1,l_2)\ge \frac{2r}{1+1/T}\ge 2r\left(1-\frac{1}{T}\right), \quad \text{as claimed.}
\end{equation*}
\end{proof}

We can now prove our first result.
\begin{thm}\label{thmfinite}
Let $R, r$ be positive real numbers, and let $\mathcal{A}$ be a point set in the $xy$ plane, such that the distance between any two points in $\mathcal{A}$ is at least $2r$, and
$\mathcal{A}$ is contained in the ball of radius $R$ centered at the origin. Then, for every $0<\epsilon\le 8r^2/R^4$, there exists a collection of congruent cylinders $\{C_i\}$, all of radius $r(1-\epsilon)$,  with the following properties:
\begin{align*}
&\text{(a) the axis of}\,\, C_i \,\, \text{passes through}\,\, A_i\in \mathcal{A}.\\
&\text{(b) no two cylinders are parallel.}\\
&\text{(c) every two cylinders have disjoint interiors.}\\
&{\text(d)} \,\, \delta(\mathcal{A},r(1-\epsilon),R)=\delta^{\perp}(\mathcal{A},r(1-\epsilon),R).
\end{align*}
\end{thm}

\begin{proof}
Use Lemma \ref{lemma1} with $T=1/\epsilon$. For each $A_i(x_i,y_i,0)\in \mathcal{A}$, consider the line $l_i$,  which
passes through $A_i$ and has direction vector $\mathbf{v}_i=\langle y_i,-x_i,1/\epsilon\rangle$. Also, recall that $l_i^{\perp}$ is the line
passing through $A_i$ and has direction vector $\langle 0,0,1\rangle$. Note that both $l_i$ and $l_i^{\perp}$ are perpendicular to $OA_i$.

Let $C_i$ and $C_i^{\perp}$ be the cylinders of radius $r(1-\epsilon)$ having axes $l_i$ and $l_i^{\perp}$, respectively.
Parts (b) and (c) follow immediately from Lemma \ref{lemma1}.
For proving part (d), note that $C_i \cap B(R)$  and $C_i^\perp \cap B(R)$ are congruent as  solids in $\mathbf{R}^3$,  as they can be obtained from each other via a rotation having $OA_i$ as its axis. Since both $\{C_i\}$ and $\{C_i^{\perp}\}$ are packings, it follows that
\begin{equation*}
\sum_{i=1}^{|\mathcal{A}|} Vol(C_i \cap B(R))= \sum_{i=1}^{|\mathcal{A}|} Vol(C_i^{\perp} \cap B(R))
\end{equation*}
Dividing both sides by the volume of $B(R)$, part (d) follows.
\end{proof}

 Theorem \ref{thmfinite} states that for given $R$, and $r$ positive numbers, if one starts with a packing of parallel cylinders of radius $r$, then one can first replace each of these cylinders with a cylinder of radius $r(1-\epsilon)$ and then rotate these thinner cylinders slightly such that no two of the perturbed cylinders are parallel; the resulting arrangement is still a packing, and the local density with respect to $B(R)$ is the same as the local density of the packing with thinner parallel cylinders.

Note that in our construction, the direction vectors of the axes of the nonparallel cylinders depend on $R$, the radius of the
circle centered at the origin that contains the set $\mathcal{A}$. Indeed, $\epsilon$ must be no greater than $8r^2/R^4$ for the construction to work. Hence, the implicit assumption is that set $\mathcal{A}$ is bounded.

If we want to produce a cylinder packing with high \underline{global density}, the set $\mathcal{A}$ has to be unbounded, and therefore, we must use a different approach.
We will present such a construction in the next section.

\section{{\bf An infinite packing with density $1/2$}}

Let $D(R)$ be the disk in $\mathbf{R}^2$ of radius $R$ and centered at the origin.

\begin{df}
Let $\mathcal{A}$ be a set of points in the plane. The \emph{point density} of $\mathcal{A}$ is defined as
\begin{equation}
\delta(\mathcal{A}, R)=\lim_{R\rightarrow\infty} \frac{|\mathcal{A} \cap D(R)|}{\pi R^2},
\end{equation}
provided the limit exists.
\end{df}

We next consider a special point set in the $xy$ plane, which we will also denote by $\mathcal{A}$.

For every positive integer $d\ge 32$, let $m$ be the unique integer such that $2^m\le d <2^{m+1}$. Define
\begin{equation*}
\theta= \frac{2\pi}{6\cdot 2^m}=\frac{\pi}{3\cdot 2^m}, \quad \text{and} \quad \mathcal{A}_d=\left\{(d\cos{k\theta}, d\sin{k\theta},0)\,: \, k=1,2,\ldots, 6\cdot 2^m\right\}.
\end{equation*}
Finally, set
\begin{equation}\label{defA}
\mathcal{A}=\bigcup_{d=32}^{\infty} \mathcal{A}_d.
\end{equation}
\begin{figure}[ht]
\centering
\includegraphics[width=0.47\linewidth]{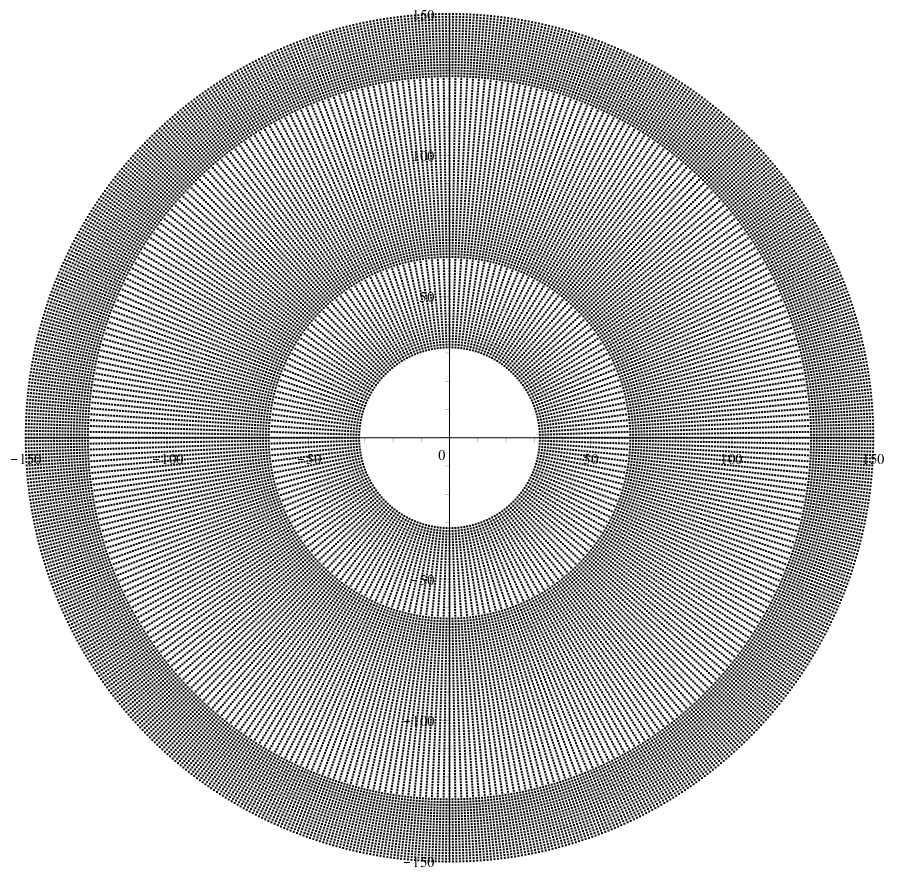}
\caption{$\bigcup\limits_{d=32}^{150}\mathcal{A}_d$ as a subset of the set $\mathcal{A}$ }
\label{fig1label}
\end{figure}
\begin{lemma}\label{2overpi}
The distance between any two distinct points of $\mathcal{A}$ is at least $1$, and its point density, regarded as a subset of $\mathbf{R}^2$, equals $2/\pi$.
\end{lemma}
\begin{proof}
Let $A_1,A_2$ be two distinct points in $\mathcal{A}$. Denote $d_1=\|OA_1\|$, $d_2=\|OA_2\|$. If $d_1\neq d_2$ then $\|A_1A_2\|\ge |d_2-d_1|\ge 1$ follows from the triangle inequality and the fact that $d_1$ and $d_2$ are integers.

If $d_1=d_2$, then $\angle A_1OA_2\ge (\pi/3\cdot 2^m)$, where $m$ is the positive integer for which $2^m\le d_1=d_2<2^{m+1}$. It follows that
\begin{equation*}
\|A_1A_2\|\ge 2d_1\sin\left(\frac{\pi}{3\cdot 2^{m+1}}\right)\ge 2^{m+1}\sin\left(\frac{\pi}{3\cdot 2^{m+1}}\right)\ge 1,
\end{equation*}
since the function $x\sin\left(\frac{\pi}{3x}\right)$ is strictly increasing, and already for $x=2$ it takes value $1$.

Let $m\ge 5$ be an integer and let $D(2^m)$ be the disk with center $O$ and radius $2^m$. The cardinality of the set $\mathcal{A} \cap D(2^m)$ is equal to
\begin{equation*}
6\cdot 32^2+6\cdot 64^2+\ldots +6\cdot (2^{m-1})^2+6\cdot 2^m= 6\cdot(2^{10}+2^{12}+\ldots+2^{2m-2}+2^m)=2^{2m+1}+6\cdot 2^m-2^{11}.
\end{equation*}
Hence, the point density of $\mathcal{A}$ equals
\begin{equation*}
\lim_{m\rightarrow \infty} \frac{2^{2m+1}+6\cdot 2^m-2^{11}}{\pi\cdot 2^{2m}}=\frac{2}{\pi}.
\end{equation*}
\end{proof}

\subsection{\bf The construction.}

Fix a real constant  $L\ge 6$ and define $K=\sqrt{L^2-1}$. Note that
\begin{equation}\label{LK}
\frac{L^2}{K^2}<1.03 \quad \text{and} \quad K<L<2K.
\end{equation}
The exact value of $L$ is not important as long as inequalities \eqref{LK} are satisfied. The reason for these choices of $L$ and $K$ is going to become clear
soon.

For a given point $A_i(x_i,y_i, 0)\in \mathcal{A}$ we denote $d_i= \|OA_i\| =\sqrt{x_i^2+y_i^2}$. Let $l_i^{\perp}$ be the line that passes through $A_i$ and is orthogonal to the $xy$ plane.

We associate with $A_i$ another line, $l_i$, in the following way:
\begin{itemize}
\item{the line $l_i$ passes through $A_i$}
\item{the direction of $l_i$ is given by the vector $\mathbf{v}_i=\langle y_i, -x_i, Kd_i+L\rangle.$}
\end{itemize}
The selection of the direction vectors is the critical part of the construction. Let us look at our choice for $\mathbf{v}_i$ a bit more closely.

First, notice that $\mathbf{v}_i\cdot \overrightarrow{OA_i} =0$, so $l_i$ is obtained by rotating $l_i^{\perp}$ around $A_i$ in the plane that passes through $A_i$ and is orthogonal to $\overrightarrow{OA_i}$.

Second, let $\beta(d_i)$ be the tangent function of the angle formed by the line $l_i$ with the $xy$ plane. Then
\begin{equation}\label{betadef}
\beta(d_i)= \frac{Kd_i+L}{\sqrt{x_i^2+y_i^2}}=\frac{Kd_i+L}{d_i}=K+\frac{L}{d_i}.
\end{equation}
Hence, $\beta(d_i)$ gets smaller as $d_i$ increases. In particular, for any two distinct points $A_i, A_j \in \mathcal{A}$ such that $d_i<d_j$ we have $\beta(d_i)>\beta(d_j)$. This means that lines $l_i$ and $l_j$ cannot be parallel. Also, if $A_i$ and $A_j$ are two distinct points in $\mathcal{A}$ with $d_i=d_j$, then it is still impossible for $l_i$ and $l_j$ to be parallel.

Indeed, in this case $\|\mathbf{v}_i\|=\|\mathbf{v}_j\|$, which implies $\mathbf{v}_i=\mathbf{v}_j$, that is, $\langle y_i,-x_i, Kd_i+L\rangle=\langle y_j,-x_j, Kd_j+L\rangle$. But this forces $A_i=A_j$, contradiction.

At this point we have the lines $l_i$, no two of which are parallel. These lines are going to be the axes of the cylinders in the packing we are planning to construct. The delicate part is to prove a lower bound for the distance between any two such lines.

Our choice of the set $\mathcal{A}$ is very similar to those in Kuperberg \cite{KK}, and Graf and Paukowitsch \cite{GP} papers, although our set has higher point density.
Also, in both those papers, $\beta(d)$ is a strictly decreasing function of $d$.

Kuperberg takes $\beta : [2,\infty)\rightarrow (1,\sqrt{3}\,]$, to be some strictly decreasing function of $d$, such that $\beta(2)=\sqrt{3}$, and $\lim\limits_{d\rightarrow \infty}\beta(d)=1$. So the lines intersecting the $xy$ plane close to the origin form angles of about $60^{\circ}$ with the $xy$ plane, and this angle decreases to about $45^{\circ}$ for lines intersecting the $xy$ plane far from the origin.

Next, Kuperberg proves that the distance between any two lines is at least $\pi/8\sqrt{3}=0.226\ldots$, and shows that this eventually leads to a cylinder packing of density $\pi^2/576=0.0171\ldots$.

Graf and Paukowitsch select their $\beta(d)=2.0896/\arctan(d)$, for $d\ge 1$. For this explicit choice of $\beta(d)$, they prove that the distance between any two lines is at least $1$, which is optimal, and then present a cylinder packing with density $5/12= 0.4166\ldots$.

While this result is much better than Kuperberg's original estimate, their choice of the function $\beta(d)$ leads to some awkward computations, which require heavy use of complicated Taylor series estimates. On the other hand, our definition \eqref{betadef} of $\beta(d)$ generates much simpler algebra and eventually helps us produce a better result.

We  will need the following result later on.
\begin{lemma}\label{lemmacos}
Let $L\ge 6$ be a fixed real number, and let $K=\sqrt{L^2-1}$. Then for every $x\ge 32$, the following inequality holds true
\begin{equation}
1-\cos\left(\frac{\pi}{3x}\right)\ge 1.03\cdot\frac{(x+1)^2}{2x^4}>\frac{L^2}{K^2}\cdot\frac{(x+1)^2}{2x^4}.
\end{equation}
\end{lemma}
\begin{proof}
Consider the function
\begin{equation}
\phi(x)= \frac{2x^4}{(x+1)^2}\cdot \left(1-\cos\left(\frac{\pi}{3x}\right)\right).
\end{equation}
It is easy to show that $\phi(x)$ is increasing for $x\ge 32$, and that $\lim\limits_{x\rightarrow\infty} \phi(x)=\pi^2/9=1.0966\ldots$.\\
Substituting $x=32$, we get $\phi(32)=1.03017\ldots$, as claimed. The second inequality is an immediate consequence of \eqref{LK}.
\end{proof}
\subsection{\bf A lower bound for the distance between two lines}
We will next study the most critical part of our construction, namely providing a lower estimate for the distance between any two lines.
\vspace{-0.3cm}
\begin{lemma}\label{mainlemma}
Let $A_1=(x_1, y_1, 0)$ and $A_2=(x_2, y_2, 0)$ be two distinct points in $\mathcal{A}$ as defined in \eqref{defA}. For $i=1, 2$, let $d_i=\sqrt{x_i^2+y_i^2}$ and $\mathbf{v}_i= \langle y_i, -x_i, Kd_i+L\rangle$, with $L$, $K$ satisfying \eqref{LK}. For $i=1, 2$, let $l_i= A_i + t \mathbf{v}_i$ be the line passing through point $A_i$ and having direction vector $\mathbf{v}_i$.   Then
\begin{equation}\label{dist}
dist(l_1, l_2)\ge 1.
\end{equation}
\end{lemma}
\begin{proof}
Note the result above is optimal, as there are infinitely many pairs of points $A_1$, $A_2$ in $\mathcal{A}$ with $\|A_1A_2\|=1$.
A key role in the sequel is played by the cosine of the angle determined by the vectors $\overrightarrow{OA_1}$ and $\overrightarrow{OA_2}$. Thus we introduce
the following notation:
\begin{equation*}
c=\cos(\angle A_1OA_2).
\end{equation*}
Note that
\begin{equation}\label{cd1d2}
c=\frac{\overrightarrow{OA_1}\cdot \overrightarrow{OA_2}}{\|OA_1\|\|OA_2\|}=\frac{x_1x_2+y_1y_2}{d_1d_2}\quad \text{from which}\quad  x_1x_2+y_1y_2= c\,d_1d_2.
\end{equation}
Again, we use the formula that gives the distance between two nonparallel lines in space:
\begin{equation}\label{distl1l2}
dist(l_1,l_2)=\frac{|\overrightarrow{A_1A_2}\cdot(\mathbf{v}_1 \times\mathbf{v}_2)|}{\|\mathbf{v}_1 \times \mathbf{v}_2\|}.
\end{equation}
As before, we will compute the numerator and denominator separately. For the numerator, we have
\begin{align*}
\overrightarrow{A_1A_2}\cdot(\mathbf{v}_1 \times\mathbf{v}_2)&=
\begin{vmatrix}
x_2-x_1 & y_2-y_1 & 0 \\
y_1 &  -x_1& Kd_1+L \\
y_2 & -x_2 & Kd_2+L
\end{vmatrix}=\\
&=-(Kd_1+L)
\begin{vmatrix}
x_2-x_1 & y_2-y_1 \\
y_2 & -x_2
\end{vmatrix}+
(Kd_2+L)
\begin{vmatrix}
x_2-x_1 & y_2-y_1 \\
y_1 & -x_1
\end{vmatrix}=\\
&= (Kd_1+L)\left( x_2^2+y_2^2-x_1x_2-y_1y_2\right)+(Kd_2+L)\left( x_1^2+y_1^2-x_1x_2-y_1y_2\right)=\\
&= (Kd_1+L)(d_2^2-c\,d_1d_2)+(Kd_2+L)(d_1^2-c\,d_1d_2).\notag\\
\end{align*}
After rearranging with respect to the powers of $1-c$ we finally obtain
\begin{equation}\label{top}
\boxed{\overrightarrow{A_1A_2}\cdot(\mathbf{v}_1 \times\mathbf{v}_2)= (1-c)d_1d_2(Kd_1+Kd_2+2L)+L(d_2-d_1)^2}.
\end{equation}
As for the denominator of \eqref{distl1l2}, similar computations give that
\begin{align*}
\|\mathbf{v}_1 \times \mathbf{v}_2\|^2 &= \|\mathbf{v}_1\|^2 \|\mathbf{v}_2\|^2- (\mathbf{v}_1\cdot \mathbf{v}_2)^2 = \\
&= \left(y_1^2 + x_1^2 + (Kd_1 + L)^2\right)\left(y_2^2 + x_2^2 + (Kd_2 + L)^2\right) - \left(y_1y_2 + x_1x_2 + (Kd_1 + L)(Kd_2 + L)\right)^2 = \\
&= \left(d_1^2 + (Kd_1 + L)^2\right)\left(d_2^2 + (Kd_2 + L)^2\right) - \left(c\:d_1d_2 + (Kd_1 + L)(Kd_2 + L)\right)^2 = \\
&= d_1^2d_2^2 + d_1^2(Kd_2 + L)^2 + d_2^2 (Kd_1 + L)^2 - c^2 d_1^2d_2^2 - 2c\:d_1d_2(Kd_1 + L)(Kd_2 + L).
\end{align*}
After arranging with respect to the powers of $(1-c)$ we obtain
\begin{equation*}
\|\mathbf{v}_1 \times \mathbf{v}_2\|^2 = -(1-c)^2 d_1^2d_2^2+2(1-c)d_1d_2\left[(K^2+1)d_1d_2+L^2+KL(d_1+d_2)\right]+ L^2(d_2-d_1)^2,
\end{equation*}
and taking into account that $K^2+1=L^2$ we end up with
\begin{equation}\label{bottom}
\boxed{\|\mathbf{v}_1 \times \mathbf{v}_2\|^2 = -(1-c)^2 d_1^2d_2^2+2(1-c)d_1d_2\left[L^2(1+d_1d_2)+KL(d_1+d_2)\right]+ L^2(d_2-d_1)^2}
\end{equation}
Since we intend to show that $|\overrightarrow{A_1A_2}\cdot(\mathbf{v}_1 \times\mathbf{v}_2)|\ge \|\mathbf{v}_1 \times \mathbf{v}_2\|$ it is natural to consider the difference

\begin{equation*}
\boxed{\Delta:=|\overrightarrow{A_1A_2}\cdot(\mathbf{v}_1 \times\mathbf{v}_2)|^2 - \|\mathbf{v}_1 \times \mathbf{v}_2\|^2}
\end{equation*}
Using equalities \eqref{top} and \eqref{bottom} we obtain
\begin{align*}
\Delta&=(1-c)^2d_1^2d_2^2(Kd_1+Kd_2+2L)^2+2(1-c)d_1d_2(KLd_1+KLd_2+2L^2)(d_2-d_1)^2+L^2(d_2-d_1)^4+\\
&+(1-c)^2d_1^2d_2^2- 2(1-c)d_1d_2\left[L^2(1+d_1d_2)+KL(d_1+d_2)\right]- L^2(d_2-d_1)^2,
\end{align*}
which after some simplifications can be written as
\begin{empheq}[box=\widefbox]{align}\label{delta}
\Delta&=(1-c)^2d_1^2d_2^2\left[ 1+(Kd_1+Kd_2+2L)^2\right]+ L^2(d_2-d_1)^2\left[(d_2-d_1)^2-1\right]+\notag\\
&+2(1-c)d_1d_2\left[KL(d_1+d_2)((d_2-d_1)^2-1)+L^2(2d_2^2-5d_1d_2+2d_1^2-1)\right]
\end{empheq}
\begin{obs}\label{dinteger}
Note that by our selection of the set $\mathcal{A}$ we have that  $L^2(d_2-d_1)^2\left[(d_2-d_1)^2-1\right]\ge 0$.
This is indeed the case, since for any point $A_i\in \mathcal{A}$, the distance $\|OA_i\|=d_i$ is a positive integer.
\end{obs}
Hence,  $L^2(d_2-d_1)^2\left[(d_2-d_1)^2-1\right]\ge 0$, and obviously $1+(Kd_1+Kd_2+2L)^2> K^2(d_1+d_2)^2$ and $1-c\ge 0$,
from which it immediately follows that
\begin{align*}
\Delta\ge (1-c)^2d_1^2d_2^2K^2(d_1+d_2)^2+2(1-c)d_1d_2\left[KL(d_1+d_2)((d_2-d_1)^2-1)+L^2(2d_2^2-5d_1d_2+2d_1^2-1)\right]
\end{align*}
The right-hand side becomes $0$ if $1-c=0$, and therefore in this case $\Delta\ge 0$, done.
Assume that $c<1$, and let $\Delta'$ be the quantity obtained after dividing the right hand term by $2(1-c)d_1d_2$:
\begin{equation}\label{Deltaprime}
\boxed{\Delta'= \frac{K^2}{2}(1-c)d_1d_2(d_1+d_2)^2+KL(d_1+d_2)((d_2-d_1)^2-1)+L^2(2d_2^2-5d_1d_2+2d_1^2-1)}
\end{equation}
We plan to prove that $\Delta'\ge 0$.\\
Clearly, $\Delta'$ is symmetric in respect to $d_1$ and $d_2$, so we can assume without loss of generality that $d_1\leq d_2$. We also know that there exists an $m\ge 5$ such that $2^m\leq d_2 <2^{m+1}$.

We will now prove the following statement:
\begin{equation}\label{ltp}
1-c\ge \frac{L^2}{K^2}\cdot \frac{(d_2+1)^2}{2d_2^4}.
\end{equation}
Since $2^m\le d_2 <2^{m+1}$ it follows by our choice of the set $\mathcal{A}$ that $\angle A_1OA_2\ge \pi/(3\cdot 2^m)\ge \pi/(3\cdot d_2)$, from which
\begin{equation*}
1-c\ge 1-\cos\left(\frac{\pi}{3\cdot 2^m}\right)\ge 1-\cos\left(\frac{\pi}{3\cdot d_2}\right) \ge \frac{L^2}{K^2}\cdot \frac{(d_2+1)^2}{2d_2^4},
\end{equation*}
where the second inequality is due to the fact that the function $1-\cos\left(\frac{1}{x}\right)$ is a decreasing function and the third inequality follows from Lemma \ref{lemmacos} with $x=d_2$, and
thus, \eqref{ltp} is true.\\
Looking at $\Delta'$ we can see that the first term is clearly non-negative and so we can plug in the lower bound we obtained in $(20)$ to get
\begin{equation*}
\Delta'\ge \frac{K^2}{2}\frac{L^2}{K^2}\cdot \frac{(d_2+1)^2}{2d_2^4}d_1d_2(d_1+d_2)^2+KL(d_1+d_2)((d_2-d_1)^2-1)+L^2(2d_2^2-5d_1d_2+2d_1^2-1).
\end{equation*}
We know that $d_1\leq d_2$ and since these are integers there exists a nonnegative integer $n$ so that $d_2=d_1+n$. Next, we substitute this expression into our inequality and simplify to obtain the following lower bound for $\Delta'$:
\begin{equation}
\frac{L^2(d_1+n+1)^2d_1(2d_1+n)^2}{4(d_1+n)^3}+KL(2d_1+n)(n^2-1)+L^2(2d_1^2-5d_1(d_1+n)+2(d_1+n)^2-1).
\end{equation}
Clearly, if we show that $(21)$ is nonnegative, then so is $\Delta'$. Multiply $(21)$ by $4(d_1+n)^3$ as that does not change the sign of the expression. After factoring and tidying up the expression we are left with a quartic polynomial in $d_1$:
\begin{align*}
&(8Kn^2-4Ln-8K+8L)d_1^4+(28Kn^3-3Ln^2-28Kn+16Ln)d_1^3+\\
+&(36Kn^4+14Ln^3-36Kn^2+10Ln^2-8Ln)d_1^2+(20Kn^5+21Ln^4-20Kn^3+2Ln^3-11Ln^2)d_1+\\
+&(4Kn^6+8Ln^5-4Kn^4-4Ln^3).
\end{align*}
To make things simpler we give numerical values to $K$ and $L$. Taking $L=7$ and $K=\sqrt{49-1}=4\sqrt{3}$ it would suffice to prove that the coefficients of the above polynomial are nonnegative for all $n\ge 0$. 

Indeed, we have

\begin{align*}
&32\sqrt{3}n^2-32\sqrt{3}-28n+56=n(32\sqrt{3}n-28)+(56-32\sqrt{3})\ge 0,\\
&112\sqrt{3}n^3-112\sqrt{3}n-21n^2+112n=n(112\sqrt{3}n(n-1)+(112\sqrt{3}-21)(n-1)+91)\ge 0, \\
&144\sqrt{3}n^4-144\sqrt{3}n^2+98n^3+70n^2-56n=2n(72\sqrt{3}n(n^2-1) + 49n^2 +(35 n - 28))\ge 0,\\
&80\sqrt{3}n^5-80\sqrt{3}n^3+147n^4+14n^3-77n^2=n^2(80\sqrt{3}n(n^2-1)+(147n^2-77)+14n)\ge 0,\\
&16\sqrt{3}n^6-16\sqrt{3}n^4+56n^5-28n^3=n^3(16\sqrt{3}n(n^2-1)+(56n^2-28)) \ge 0.
\end{align*}
Now it is clear that the whole expression is nonnegative. Thus $\Delta'\ge0$ and the proof is complete.
\end{proof}
\newpage
We are now ready to prove the main result of this paper.
\begin{thm}
There exists a cylinder packing $\mathcal{C}$  whose global density is $1/2$, and no two cylinders in $\mathcal{C}$ are parallel.
\end{thm}

\begin{proof}
Consider the set $\mathcal{A}$ as defined in \eqref{defA}. For each point $A_i(x_i,y_i,0)\in \mathcal{A}$, define two lines as follows:
\begin{align*}
l_i^\perp &= A_i + \langle 0,0,1\rangle t \quad \text{and}\\
l_i&=A_i+\langle y_i,-x_i, Kd_i+L\rangle t,
\end{align*}
where $d_i=\sqrt{x_i^2+y_i^2}$, $L=7$, and $K=4\sqrt{3}$.

Let $C_i^{\perp}$ and $C_i$ be the cylinders of radius $1/2$ whose axes are $l_i^{\perp}$ and $l_i$, respectively.
Then, clearly $\{C_i^{\perp}\}$ is a packing with parallel cylinders, while by Lemma \ref{mainlemma} $\{C_i\}$ is also a cylinder packing, but no two cylinders are parallel to each other. Moreover, since for every $R>0$ we have that $B(R)\cap C_i^{\perp}$ and $B(R)\cap C_i$ are congruent, it follows that
\begin{equation*}
\delta(\{C_i^{\perp}\},R)=\delta(\{C_i\},R),
\end{equation*}
and after letting $R\longrightarrow \infty$, the global densities of the two cylinder packings are also equal to each other.
\begin{equation}
\delta(\{C_i^{\perp}\})=\delta(\{C_i\}),
\end{equation}

But it is easy to see that the density of $\{C_i^{\perp}\}$ can be expressed as the product of the point density of $\mathcal{A}$ (which by Lemma \ref{2overpi} equals $2/\pi$) and the area of the perpendicular cross-section of a cylinder in $\{C_i^{\perp}\}$. Hence, the global density of $\{C_i^{\perp}\}$ equals
\begin{equation}
\delta(\{C_i^{\perp}\})=\frac{2}{\pi}\cdot \pi\left(\frac{1}{2}\right)^2=\frac{1}{2}.
\end{equation}

Combining the last two equalities, we obtain the desired result.

\end{proof}

\section{\bf{Conclusions and directions for future study}}

We constructed a cylinder packing with global density $1/2$, with no two cylinders being parallel. The natural question is whether we can do better.

In fact, K. Kuperberg \cite{KK} conjectured that there exist such packings with density arbitrarily close to $\pi/\sqrt{12}$.
In light of our result in Theorem \ref{thmfinite}, this may very well be the case.

The weak part of our proof is the choice of the point set $\mathcal{A}$; this set has point density of only $2/\pi=0.6366\ldots$.  In comparison, the point set
\begin{equation}
\Lambda=\{(m/2, n\sqrt{3}/2): m+n \,\,\text{even},\,\ m, n \,\,\text{integers}\}
\end{equation}
has point density $2/\sqrt{3}=1.1647\ldots$, and the minimum distance between any two points of $\Lambda$ is still equal to $1$.
Of course, $\Lambda$ is the unit equilateral triangle lattice; hence by Thue's result, it is the densest planar point set with this property.

So why not use $\Lambda$ instead of $\mathcal{A}$? The answer is simple: for every point $A_i\in \mathcal{A}$, the distance $d_i=\|OA_i\|$ is an integer.
This is not the case with the lattice $\Lambda$.

We used this fact in estimating the quantity $\Delta$ defined in \eqref{delta}; see also observation \ref{dinteger}. To be precise,
the term $L^2(d_2-d_1)^2\left[(d_2-d_1)^2-1\right]$ is always nonnegative if $d_i$ are all integers. This allowed us to ignore this term altogether
and, in the process, reduce the problem to a \emph{linear} inequality in $1-c$; see the step that lead to the definition of $\Delta'$ in \eqref{Deltaprime}.

But of course, if we choose $\Lambda$ instead of $\mathcal{A}$, $L^2(d_2-d_1)^2\left[(d_2-d_1)^2-1\right]$ could very well be negative. And then, showing that $\Delta\ge 0$ becomes
a \emph{quadratic} inequality in $(1-c)$. Not only is this inequality more difficult to work with, but several computer experiments indicated that it may not even be true.

One final remark regarding our cylinder packing.

The angle formed by the line $l_i$ with the $xy$ plane is $\arctan(K+L/d_i)$. This means that lines that intersect the $xy$ plane closest to the origin have slopes of
$K+L/32$, while the lines that intersect the $xy$ plane far from the origin have slopes close to $K$. Once $L$ (and therefore $K$) is fixed, this creates a cone having $Oz$ as its axis, which is completely free of any cylinders. In other words, there are arbitrarily large holes in our packing. We measure the size of the hole by the radius of the largest sphere
that can fit in that region.

One interesting question is whether a nonparallel cylinder packing exists, all of its holes being of bounded size.

\end{document}